\documentclass[12pt,a4paper,leqno,verbatim]{amsart}
\newcounter{minutes}
\divide\time by 60
\newcounter{hours}
\multiply\time by 60 \addtocounter{minutes}{-\time}

\usepackage{amssymb}
\usepackage{hyperref}
\usepackage{graphicx}
\usepackage{subfigure}
\date{}
\newfont{\cyrilic}{wncyr10 scaled 1000}

\title[{Inequalities for means of two mean variables}]
{Inequalities for means of two mean variables}

\author[L. Yin]{Li Yin}
\address{Department of Mathematics, Binzhou University, Binzhou City, Shandong Province, 256603, China}
\email{yinli\_79@163.com}

\author[B. A. Bhayo]{Barkat Ali Bhayo}
\address{Education and teaching department, 50101 Mikkeli, Finland}
\email{bhayo.barkat@gmail.com}

\newcommand{\comment}[1]{}

\swapnumbers
\theoremstyle{plain}

\newtheorem{theorem}[equation]{Theorem}
\newtheorem{lemma}[equation]{Lemma}
\newtheorem{definition}[equation]{Definition}

\newtheorem{corollary}[equation]{Corollary}

\numberwithin{equation}{section}

\begin{document}

\allowdisplaybreaks

\begin{abstract}
Motivated by the work of Anderson, Vamanamurthy and Vuorinen \cite{avv}, 
in this paper authors study the log-convexity and log-concavity of Power mean, Identric mean, weighted Power mean, Lehmer mean, Modified Alzer mean, and establish the relation of these means with each other.
\end{abstract}

\vspace{.5cm}

\maketitle



\def\thefootnote{}
\footnotetext{ \texttt{\tiny File:~\jobname .tex,
          printed: \number\year-0\number\month-\number\day,
          \thehours.\ifnum\theminutes<10{0}\fi\theminutes}
} \makeatletter\def\thefootnote{\@arabic\c@footnote}\makeatother


\section{introduction}

For two positive real numbers $a$ and $b$, the Arithmetic, Geometric, Harmonic, Logarithmic, Identric, Lehmer, Modified Alzer, Power mean of order $t\in\mathbb{R}$,  and the weighted Power mean are respectively defined by
$$A(a,b)=\frac{a+b}{2},\quad G(a,b)=\sqrt{ab},$$
$$H(a,b)=\frac{1}{A(1/a,1/b)},$$
$$L(a,b) = \frac{{a - b}}{{\log a - \log b}},a\neq b,L(a,a)=a,$$
$$I(a,b)=
\frac{1}{e}\left( {\frac{{a^a }}{{b^b }}} \right)^{1/(a-b)}
,\quad a\neq b,$$
$$L_p (a,b) = \frac{{a^p  + b^p }}{{a^{p - 1}  + b^{p - 1} }},\quad p > 0,
$$

$$
J_p (a,b) = \frac{{p + 1}}{p}\frac{{a^{p + 1}  - b^{p + 1} }}{{a^p  - b^p }}, \quad p \ne 0,a \ne b,
$$

$$
M_p (a,b) = \left\{ \begin{array}{l}
 \left( {\frac{{a^p  + b^p }}{2}} \right)^{1/p} ,\quad p \ne 0, \\
 \sqrt {ab} , \quad p = 0, \\
 \end{array} \right.
$$
$$
M_p (\omega ,a,b) = \left( {\frac{{a^p  + \omega b^p }}{{1 + \omega }}} \right)^{1/p} , \quad p > 0.
$$
Since last few decades the inequalities involving these means have been studied extensively by numerous authors. For the historical background, generalized, and their connection with elementary functions and with each other, we refer the reader to \cite{alzer1, alzer2,newmean,carlson,mit,ns0206,ns1004a,ns1004}.
 
For the following definition see \cite{avv}.
\begin{definition}
Let $f:I_0\to (0,\infty)$ be continuous, where $I$ is a sub-interval of $(0,\infty)$. Let $M$ and $N$ be two any mean functions. We say that the function
$f$ is $MN$-convex (concave) if
$$f (M(x, y)) \leq (\geq)
N(f (x), f (y)) \,\, \text{ for \,\, all} \,\, x,y \in I_0\,.$$
\end{definition}
In \cite{avv}, Anderson, Vamanamurthy and Vuorinen studied the convexity and concavity properties of a function $f$ with respect to two mean values, and gave the following result:
\begin{lemma}\cite[Theorem 2.4]{avv} Let 
$f : I_0\to (0,\infty)$ be a differentiable.
In parts (4)–(9), let $I_0 = (0, b),\, 0<b<\infty$. Then
\begin{enumerate}
\item $f$ is $AA$-convex (concave) if and only if $f'(x)$ is increasing (decreasing),
\item $f$ is $AG$-convex (concave) if and only if $f'(x)/f (x)$ is increasing (decreasing),
\item $f$ is $AH$-convex (concave) if and only if $f'(x)/f (x)^2$ is increasing (decreasing),
\item $f$ is $GA$-convex (concave) if and only if $xf'(x)$ is increasing (decreasing),
\item $f$ is $GG$-convex (concave) if and only if $xf'(x)/f (x)$ is increasing (decreasing),
\item $f$ is $GH$-convex (concave) if and only if $xf'(x)/f (x)^2$ is increasing (decreasing),
\item $f$ is $HA$-convex (concave) if and only if $x^2f'(x)$ is increasing (decreasing),
\item $f$ is $HG$-convex (concave) if and only if $x^2f'(x)/f (x)$ is increasing (decreasing),
\item $f$ is $HH$-convex (concave) if and only if $x^2f'(x)/f (x)^2$ is increasing 
(decreasing).
\end{enumerate}
\end{lemma}
After the publication \cite{avv}, many authors have studied generalized convexity. For a partial survey of the recent results, see  \cite{avz}. In \cite{by1}, the following results appeared. 
Motivated by the above work of Anderson et al. \cite{avv}, the authors of this paper studied the convexity and concavity properties of a function with respect to Logarithmic, Identric mean and gave the following results..

\begin{lemma}\label{lem2.1}\cite{by1}
Let $f:(0,1)\to(0,\infty)$ be a continuous, then 
\begin{enumerate}
\item $f$ is $LL$-convex (concave) if $f$ is increasing and $\log$-convex (concave),
\item $f$ is $AL$-convex (concave) if $f$ is increasing and $\log$-convex (concave).
\end{enumerate}
\end{lemma}

\begin{lemma}\label{lem2.2}\cite[Theorem 1]{by2}
Let $f:I \rightarrow (0,\infty)$ and $I\subseteq(0,\infty).$
Then the following inequality holds true:
$$  I(f(x),f(y))\geq f(I(x,y))$$
$$( I(f(x),f(y))\leq f(A(x,y)))$$
If the function $f(x)$ is a continuously differentiable, increasing and log-convex
(concave).
\end{lemma}

In this paper authors make a contribution to the topic by giving the following theorems. 
\begin{theorem}\label{theorem1}
For all $a,b\in(0,\infty)$ fixed, if $p,q>0$ or $p,q\leq0$ and $b\geq a$, then the following hold true:\\
(1)~~$L\left( {M_p (a,b),M_q (a,b)} \right) \le M_{L(p,q)} (a,b);$\\
(2)~~$L\left( {M_p (a,b),M_q (a,b)} \right) \le M_{A(p,q)} (a,b);$\\
(3)~~$I\left( {M_p (a,b),M_q (a,b)} \right) \le M_{A(p,q)} (a,b).$\\
\end{theorem}

\begin{theorem}\label{theorem2}
For all $a,b\in(0,\infty)$ fixed and $b\geq a$, if $p,q>0$, then the following hold true:\\
(1)~~$L\left( {L_p (a,b),L_q (a,b)} \right) \geq L_{L(p,q)} (a,b);$\\
(2)~~$L\left( {L_p (a,b),L_q (a,b)} \right) \geq L_{A(p,q)} (a,b);$\\
(3)~~$I\left( {L_p (a,b),L_q (a,b)} \right) \geq L_{I(p,q)} (a,b).$\\
\end{theorem}

\begin{theorem}\label{theorem3}
For $a\geq b\geq 1$, $p,\omega,\nu>0$, the following hold true:\\
(1)~~$L\left( {M_p (\omega ,a,b),M_p (\nu,a,b)} \right) \le M_p (L(\omega ,\nu),a,b),$\\
(2)~~$L\left( {M_p (\omega ,a,b),M_p (\nu,a,b)} \right) \le M_p (A(\omega ,\nu),a,b),$\\
(3)~~$I\left( {M_p (\omega ,a,b),M_p (\nu,a,b)} \right) \le M_p (A(\omega ,\nu),a,b).$\\
\end{theorem}

\section{Lemmas and proof}
In this section we give few lemmas which will be used in the proof of main result.

\begin{lemma}[{\cite[Lemma~1]{mil}}]\label{lem2.3}
The two variable power mean $M_p(a,b)$ is concave in $p$ for $p\geq1$ and convex in $p$ for $p\leq-1$.
That is, $\frac{{\partial ^2 }}{{\partial p^2 }}\left[ {M_p (a,b)} \right] \le 0$ for $p\geq1$ and $\frac{{\partial ^2 }}{{\partial p^2 }}\left[ {M_p (a,b)} \right] \geq 0
$~~~ for $p\leq-1$ with equality if and only if $a=b$.
\end{lemma}

\begin{lemma}\label{thm2.1}
For $a,b>0$, the following assertions hold true:\\
(1)~~The function $p\mapsto M_p(a,b)$ is increasing and log-concave on $p\in (0,\infty)$;\\
(2)~~If $\frac{b}{a}\geq1$, the function $p\mapsto M_p(a,b)$ is increasing and log-concave on $p\in (-\infty,0)$.
\end{lemma}
\begin{proof}
By simple computation, we easily know that $M_p(a,b)=aM_p(1,\lambda)$, where $\lambda=b/a$. So, we only need to prove that the function $p\mapsto M_p(a,b)$
satisfies the above assertions.

For the proof of part (1), we easily obtain
$$
M_{pt} (1,\lambda )^t  = M_p (1,\lambda ^t ),
$$
and
\begin{equation}\label{(2.2)}
t\log \left( {M_{pt} (1,\lambda )} \right) = \log \left( {M_p (1,\lambda ^t )} \right).
\end{equation}
Differentiate ~\eqref{(2.2)} with respect to $p$ while holding $\lambda$ fixed, we get

\begin{equation}\label{(2.3)}
t^2 \left[ {\log \left( {M_{pt} (1,\lambda )} \right)} \right]'   = \left[ {\log \left( {M_p (1,\lambda ^t )} \right)} \right]',
\end{equation}
and
\begin{equation}\label{(2.4)}
t^3 \left[ {\log \left( {M_{pt} (1,\lambda )} \right)} \right]^{''}  = \left[ {\log \left( {M_p (1,\lambda ^t )} \right)} \right]^{''}.
\end{equation}
Putting $p=1$, we have $$
t^3 \left[ {\log \left( {M_t (1,\lambda )} \right)} \right]^{''}  \le 0
$$
by Lemma~\ref{lem2.3}. Hence, we have $\left[ {\log \left( {M_t (1,\lambda )} \right)} \right]^{''}  \le 0$ for $t\geq0$.

For the proof of part (2). Let $
{f(p) = M_p (1,\lambda )}.
$ Simple computation yields
$$
f'(p) =  - \frac{1}{{p^2 }}\log \left( {\frac{{1 + \lambda ^p }}{2}} \right) + \frac{{\lambda ^p \ln \lambda }}{{p^2 \left( {1 + \lambda ^p } \right)}},
$$
and
$$
f''(p) = \frac{2}{{p^3 }}\log \left( {\frac{{1 + \lambda ^p }}{2}} \right) - \frac{{\lambda ^p \ln \lambda }}{{p^2 \left( {1 + \lambda ^p } \right)}} + \frac{{p\lambda ^p \ln ^2 \lambda  - \lambda ^p \ln \lambda  - \lambda ^{2p} \ln \lambda }}{{p^2 \left( {1 + \lambda ^p } \right)^2 }}.
$$
If $p<0$ and $\lambda>1$, we easily see that $f''(p)<0$ which implies that the function $p\mapsto M_p(a,b)$ is log-cancave on $p\in (-\infty,0)$. In addition, the increasing property of $p\mapsto M_p(a,b)$ is well-known see e.g., \cite{kua}. This completes the proof.
\end{proof}

\begin{lemma}\label{thm2.2}
For $a,b>0$ with $b\geq a$, the function $p\mapsto L_p(a,b)$ is increasing and log-convex on $p\in (0,\infty)$. In particular, for $p>1$ and $b\geq a>0$, we have
$$L_p(a,b)^2\leq L_{p+1}(a,b)L_{p-1}(a,b).$$
\end{lemma}
\begin{proof}
Since $L_p(a,b)=aL_p(1,\lambda), \lambda=\frac{b}{a}$, we only consider the function $L_p(1,\lambda)$.
Let $
{g(p) = \log\left(L_p (1,\lambda )\right)}.
$ Simple computation results in
$$
g'(p) = \frac{{\lambda ^p \ln \lambda }}{{1 + \lambda ^p }}\geq0,
$$
and
$$
g''(p) = \frac{{\lambda ^p \ln ^2 \lambda }}{{1 + \lambda ^p }} \geq 0
$$
which implies that the function $g(p)$ is increasing and log-convex on $p\in (0,\infty)$.
\end{proof}

\begin{lemma}\label{thm2.4}
For $p>0$, the function $p\mapsto M_p(\omega,a,b)$ is decreasing and log-convex (increasing and log-concave) on $\omega\in (0,\infty)$ if $b\geq a>1$ $(a\geq b >1)$. In particular for $p>0, u>1$ and 
$b\geq a>1$ $(a\geq b >1)$, we have
$$M_p(u,a,b)^2\leq (\geq) M_p(u-1,a,b)M_p(u+1,a,b).$$
\end{lemma}

\begin{proof}
Let $
k(\omega) = \log\left(M_p (\omega,a,b)\right).
$ By simple computation, we have
$$
k'(\omega ) = \frac{1}{p}\left( {\frac{1}{{1 + \omega (b/a)^p }} - \frac{1}{{1 + \omega }}} \right)
$$
and
$$
k''(\omega ) = \frac{1}{p}\left( {\frac{1}{{\left( {1 + \omega } \right)^2 }} - \frac{1}{{\left( {1 + \omega (b/a)^p } \right)^2 }}} \right).
$$
If $b\geq a>1$ $(a\geq b >1)$, we easily obtain $k'(\omega)\leq (\geq) 0$ and $k''(\omega)\geq (\leq) 0$.
This completes the proof.
\end{proof}

\begin{corollary}\label{thm2.3}
For $b>a>0$, the function $p\mapsto J_p(a,b)$ is strictly decreasing and log-convex on $p\in (0,\infty)$.
In particular, for $b>a>0$ and $p>1$, we have
$$J_p(a,b)^2\leq J_{p+1}(a,b)J_{p-1}(a,b).$$
\end{corollary}

\begin{proof}
Since $J_p(a,b)=aJ_p(1,\lambda)$, where $\lambda=\frac{b}{a}$. Now it is enough to prove that  the function $p\mapsto J_p(1,\lambda)$ is strictly decreasing and log-convex on $p\in (0,\infty)$.
Let $
h(p) = \log\left(J_p (1,\lambda )\right).
$ By simple computation, we have
$$
h'(p) =  - \frac{1}{p} + \frac{1}{{p + 1}} + \ln \lambda \left( {\frac{{\lambda ^p }}{{1 - \lambda ^p }} - \frac{{\lambda ^{p + 1} }}{{1 - \lambda ^{p + 1} }}} \right) < 0,
$$
and
$$
h''(p) = \frac{1}{{p^2 }} - \frac{1}{{\left( {p + 1} \right)^2 }} + \ln ^2 \lambda \left( {\frac{{\lambda ^p }}{{\left( {1 - \lambda ^p } \right)^2 }} - \frac{{\lambda ^{p + 1} }}{{\left( {1 - \lambda ^{p + 1} } \right)^2 }}} \right) > 0.
$$
In fact, we consider the auxiliary functions $\alpha (x) = \frac{x}{{1 - x}}$ and $\beta (x) = \frac{x}{{(1 - x)^2 }}$ for $x>1$. Simple computation yields $
\alpha '(x) = \frac{1}{{(1 - x)^2 }} > 0
$ and $
\beta '(x) = \frac{{1 + x}}{{(1 - x)^3 }}<0.
$
This implies that $\alpha(x)$ is strictly increasing and $\beta(x)$ is strictly decreasing for $x>1$.
Because of $\lambda>1$, we have $\lambda^p<\lambda^{p+1}$.
So, we get $
{\frac{{\lambda ^p }}{{1 - \lambda ^p }} < \frac{{\lambda ^{p + 1} }}{{1 - \lambda ^{p + 1} }}}
$ and $
{\frac{{\lambda ^p }}{{\left( {1 - \lambda ^p } \right)^2 }} > \frac{{\lambda ^{p + 1} }}{{\left( {1 - \lambda ^{p + 1} } \right)^2 }}}
$ This completes the proof.
\end{proof}

\bigskip

\noindent {\bf Proof of Theorem \ref{theorem1}-\ref{theorem3}.}
The proof of Theorem \ref{theorem1} follows from Lemmas \ref{thm2.1}, \ref{lem2.1} and \ref{lem2.2}. Similarly, by utilizing the Lemmas \ref{lem2.1} and \ref{lem2.2}, the proof of Theorem \ref{theorem2} and \ref{theorem3} follow from \ref{thm2.2} and \ref{thm2.4}, respectively.
$\hfill\square$

\end{document}